\documentclass[12pt]{article}

\usepackage{cite}
\usepackage{mathrsfs}
\usepackage{amsfonts}
\usepackage{amsmath}
\usepackage{amssymb}
\usepackage{graphicx}
\usepackage{color}
\usepackage{amsthm}
\usepackage{float}
\usepackage[a4paper,top=3.0cm,bottom=3.0cm,left=2.5cm,right=2.5cm]{geometry}

\newtheorem{theorem}{Theorem}[section]

\newtheorem{definition}{Definition}[section]

\newtheorem{claim}{Claim}[section]
\newtheorem{case}{Case}

\def\emptyset{\mbox{{\rm \O}}}
\def\bar{\overline}

\begin{document}
	
	\title{A Fan-type condition involving bipartite independence number for hamiltonicity in graphs\thanks{This work is supported by National Natural Science Foundation of China (No. 12271169 and 12331014) and Science and Technology Commission of Shanghai Municipality (No. 22DZ2229014) }}
\author{Hongxi Liu$^{1,}$\footnote{Email: hongxiliu1@163.com.}, \;Long-Tu Yuan$^{1,}$\footnote{Email: ltyuan@math.ecnu.edu.cn.}, \;  Xiaowen Zhang$^{1,}$\footnote{Corresponding author. Email: xiaowzhang0128@126.com.}\\
	\small $^1$School of Mathematical Sciences, Shanghai Key Laboratory of PMMP\\
	\small East China Normal University, Shanghai, 200241, China \\ 
}
\date{}

\maketitle

\begin{abstract}
The bipartite independence number of a graph $G$, denoted by $\widetilde{\alpha}(G)$, is defined as the smallest integer $q$ for which there exist positive integers $s$ and $t$ with $s + t = q + 1$, such that for any two disjoint subsets $A, B \subseteq V(G)$ with $|A| = s$ and $|B| = t$, there exists an edge between $A$ and $B$. In this paper, we prove that for a 2-connected graph $G$ of order at least three, if $\max\{d_G(x), d_G(y)\} \ge \widetilde{\alpha}(G)$ for every pair of nonadjacent vertices $x, y$ at distance two, then $G$ is hamiltonian. Moreover, we prove that if $G$ is 3-connected and $\max\{d_G(x), d_G(y)\} \ge \widetilde{\alpha}(G)+1$ for every pair of nonadjacent vertices $x, y$ at distance two, then $G$ is hamiltonian-connected. Our results generalize the recent work by Li and Liu.
	\begin{flushleft}
	\end{flushleft}
	
\end{abstract}

\section{Introduction}

We consider finite simple graphs. For any undefined terminology or notation, we refer to the books \cite{Bondy2008, West1996}.
Let $V(G)$ and $E(G)$ denote the vertex set and edge set of a graph $G$, respectively. The \emph{order} and \emph{size} of $G$ are denoted by $|G|$ and $e(G)$, respectively. Let $\delta(G)$ denote the minimum degree of a graph $G$.
Denote by $\deg_G(v)$ the minimum degree of $v$ in a graph $G$.
Denote by $K_n$ a complete graph of order $n$. $G$ denotes the complement of a graph $G$.
For two graphs $G$ and $H$, $G\vee H$ denotes the \emph{join} of $G$ and $H$, which is obtained from the disjoint
union of $G$ and $H$ by adding edges joining every vertex of $G$ to every vertex of $H$.
A Hamilton  cycle in $G$ is a cycle containing every vertex of $G$. A graph $G$ is hamiltonian if it contains a Hamilton cycle.

Dirac's classic theorem \cite{Dirac1952} from 1952 states that if $G$ is a graph of order at least three and $\delta(G)\geq n/2$, then $G$ is hamiltonian, which is probably the first nontrivial sufficient condition for a graph to be hamiltonian.
Ore's theorem \cite{Ore1960} from 1960 says that if $G$ is a graph of order at least three and $d_G(x)+d_G(y)\geq n$ for each pair of nonadjacent vertices $x, y\in V(G)$, then $G$ is hamiltonian, which generalized Dirac's theorem.
In 1984, Fan \cite{Fan1984} gave another sufficient condition for a graph to be hamiltonian.

\begin{theorem}[Fan \cite{Fan1984}]
	If $G$ is 2-connected and $\max\{d_G(x),d_G(y)\}\geq n/2$ for each pair of vertices $x,y$ with $d_G(x,y)=2$, then $G$ is hamiltonian.
\end{theorem}

Fan's result
is a significant improvement on Ore's theorem, and the degree condition stated there is called the Fan condition. For more information on some of these generalizations, we refer the reader to  \cite{Dirac1952,Fan1984,Faudree1989,Gould2014,Li2013,Li2025,Ore1960}.
Motivated by Dirac's theorem, the following notion of bipartite independence number was introduced by
McDiarmid and Yolov \cite{Mcdiarmid2017} in 2017.
\begin{definition}
	The bipartite independence number of a graph $G$, denoted by $\widetilde{\alpha}(G)$, is defined as the smallest integer $q$ for which there exist positive integers $s$ and $t$ with $s + t = q + 1$, such that for any two disjoint subsets $A, B \subseteq V(G)$ with $|A| = s$ and $|B| = t$, there exists an edge between $A$ and $B$.
\end{definition}

By considering the minimum degree and the bipartite independence number, McDiarmid and Yolov \cite{Mcdiarmid2017} proved that if $G$ is a graph of order at least three and $\delta(G)\ge \widetilde{\alpha}(G)$, then $G$ is hamiltonian.
Since for a graph $G$ with $\delta(G)\geq n/2$, and for any two disjoint subsets $A, B \subseteq V(G)$ with $|A| = 1$ and $|B| = \lfloor n/2\rfloor$, there exists an edge between $A$ and $B$. The condition in the result of McDiarmid and Yolov is weaker than that in Dirac's theorem.

A graph is called  hamiltonian-connected if between any
two distinct vertices there is a Hamilton path.
By considering degree sum conditions for any graph to be hamiltonian-connected, Ore \cite{Ore1963} proved that
if $G$ is  a graph of order at least three and $d_G(x)+d_G(y)\geq n$ for each pair of nonadjacent vertices $x, y\in V(G)$, then $G$ is hamiltonian-connected.
In 2024, based on the minimum degree and the bipartite independence number, Zhou, Broersma, Wang and Lu \cite{Zhou2024} proved that
if $G$ is a graph of order at least three and with $\delta(G)\ge \widetilde{\alpha}(G)+1$, then $G$ is hamiltonian-connected.

More recently, by combining degree sum conditions and the bipartite independence number, Li and Liu \cite {Li2025} gave sufficient conditions for hamiltonicity and hamiltonian connectedness.

\begin{theorem}[Li and Liu \cite {Li2025}]
	Let $G$ be a 2-connected graph of order at least three. If $d_G(x)+d_G(y) \ge 2\widetilde{\alpha}(G)$ for each pair of nonadjacent vertices $x, y\in V(G)$, then $G$ is hamiltonian.
\end{theorem}

\begin{theorem}[Li and Liu \cite {Li2025}]
	Let $G$ is a 3-connected graph of order at least three. If $d_G(x)+d_G(y) \ge 2\widetilde{\alpha}(G)+1$ for each pair of nonadjacent vertices $x, y\in V(G)$, then $G$ is hamiltonian-connected.
\end{theorem}

In this paper, we give sufficient Fan-type conditions involving bipartite independence number for hamiltonicity and hamiltonian connectedness.

\begin{theorem}\label{hamiltonian theorem}
	Let $G$  be a $2$-connected graph of order at least three. If $\max\{d_G(x),\allowbreak d_G(y)\}\geq \widetilde{\alpha}(G)$ for any nonadjacent vertices $x$ and $y$ with $d_G(x,y)=2$, then $G$ is hamiltonian.
\end{theorem}

\begin{theorem}\label{hamiltonian-connected theorem}
	Let $G$  be a $3$-connected graph. If $\max\{d_G(x),d_G(y)\}\geq \widetilde{\alpha}(G)$+1 for any nonadjacent vertices $x$ and $y$ with $d_G(x,y)=2$, then $G$ is hamiltonian-connected.
\end{theorem}

The bounds in Theorems~\ref{hamiltonian theorem} and~\ref{hamiltonian-connected theorem} are tight.
For Theorem~\ref{hamiltonian theorem}, consider the graph $G_1= K_{n}\vee\bar{K_{n+1}}$. It is easy to verify $\widetilde{\alpha}(G_1)=n+1$. This graph is $2$-connected and satisfies $\max\{d_{G_1}(x), d_{G_1}(y)\} = \widetilde{\alpha}(G_1) - 1$ for any nonadjacent vertex pairs $x,y$ with $d_{G_1}(x,y) = 2$. However, $G$ is not hamiltonian.
For Theorem~\ref{hamiltonian-connected theorem}, consider the graph $G_2 = K_{n}\vee\overline{K_{n}}$. It is easy to verify $\widetilde{\alpha}(G_2)=n$. This graph is $3$-connected for $n\geq 3$ and satisfies $\max\{d_{G_2}(x), d_{G_2}(y)\} = \widetilde{\alpha}(G_2)$ for any nonadjacent vertex pairs $x,y$ with $d_{G_2}(x,y) = 2$. Yet $G_2$ is not hamiltonian-connected.

The $3$-connectivity condition in Theorem \ref{hamiltonian-connected theorem} is necessary. Let $a\geq 5$ be an integer. Consider the graph $G = (K_{a-2}\cup K_1)\vee K_2\triangleq A\vee B$. Clearly, $G$ is not 3-connected, $\widetilde{\alpha}(G)= 2$ and $\max\{\deg_G(x),\deg_G(y)\}=a-2$ for any nonadjacent vertices $x,y$ at distance 2. We have $\max\{\deg_G(x),\deg_G(y)\}\geq$ $\widetilde{\alpha}(G)+1.$ Since there exists no Hamilton path with endpoints in $B$, $G$ is not hamiltonian-connected.

We organize the remainder of this paper as follows: Section \ref{sec2} presents the proofs of Theorem
\ref{hamiltonian theorem}, while Section \ref{sec3} focuses on the proof of Theorem \ref{hamiltonian-connected theorem}.

\section{Proof of Theorem \ref{hamiltonian theorem}}\label{sec2}

Before starting to prove Theorem \ref{hamiltonian theorem}, we need some definitions and notations.
For a vertex $u \in V(G)$ and a subgraph $H \subseteq G$, we write $N_H(u)$ for the set of neighbors of $u$ that are contained in $V(H)$. Given a subset $S \subseteq V(G)$, define
$N_G(S) = \bigcup_{x \in S} N_G(x) \setminus S$, and $N_H(S) = N_G(S) \cap V(H)$. 
The subgraph of $G$ induced by a vertex subset $S$ is denoted by $G[S]$, and we write $G - S$ for the induced subgraph $G[V(G) \setminus S]$. A clique in a graph is an induced subgraph such that any two vertices in this subgraph are adjacent. For disjoint vertex subsets $A,B\in V(G)$, denote by $[A,B]$ the edge set with one terminal vertex in $A$, and the other in $B$.

A \emph{path} in a graph is a sequence of distinct vertices $v_0, v_1, \dots, v_k$ such that $v_{i-1}v_i \in E(G)$ for all $i = 1, \dots, k$.  A \emph{segment} refers to a subpath of a path, i.e., a consecutive subsequence $v_i, v_{i+1}, \dots, v_j$ of a path $v_0, v_1, \dots, v_k$ with $0 \le i < j \le k$. Let $P$ be an oriented $(u, v)$-path. We use $P[x,y]$ to denote the segment of $P$ between two  vertices $x,y\in V(P)$, $\overrightarrow{P}[x,y]$ to denote the segment of $P$ from $x$ to $y$ which follows the orientation of $P$, and $\overleftarrow{P}[x,y]$ to denote the opposite segment of $P$ from $x$ to $y$. Moreover, for $x\neq v$, denote by $x^{+}$ the immediate successor on $P$; and for $x\neq u$, denote by $x^{-}$ the predecessor on $P$. For $S\subseteq V(P)$, let $S^{+}=\{x^{+}: x\in S\setminus\{v\}\}$ and $S^{-}=\{x^{+}: x\in S\setminus\{u\}\}$. A {\em matching} in a graph is a set of pairwise nonadjacent edges. 

\begin{proof}[\bf Proof of Theorem \ref{hamiltonian theorem}]

Suppose that $G$ is a graph satisfying the given condition and $G$ has no Hamilton cycle. We shall arrive at a contradiction. Let $P =  v_1,v_2, \ldots, v_m$ be a longest path in $G$ of length $m-1$, chosen so that $d(v_1) + d(v_m)$ is as large as possible. 
Then further we suppose $G$ has no cycle of length $m$. In fact, if $G$ has a cycle of length $m$, then either $G$ is hamiltonian or $G$ has a path of length $m$, both cases lead to a contradiction. Without loss of generality, suppose $d_G(v_1)\leq d_G(v_m)$.

\begin{claim}\label{vd}
	$d_G(v_1)\ge \widetilde\alpha(G)$.
\end{claim}

\begin{proof} Suppose to the contrary, $d_G(v_1)< \widetilde\alpha(G)$. Since $G$ is 2-connected, $v_1$ has a neighbor other than $v_2$. Let $v_\ell$ be a neighbor of $v_1$. Choose such that $\ell$ is as large as possible. Note that
	\[v_{\ell-1},\overleftarrow{P}[v_{\ell-1},v_1],v_1,v_\ell,\overrightarrow{P}[v_\ell,v_m],v_m\]
	is a path of length $m-1$ with endpoints $v_{\ell-1}$ and $v_m$. By the choice of $P$, we have $d_G(v_{\ell-1})\le d_G(v_1)<\widetilde\alpha(G)$.
	Since $\max\{d_{G}(x),d_G(y)\}\geq \widetilde{\alpha}(G)$ for any $x,y\in V(G)$ with $d_G(x,y)=2$, $v_{\ell-1}$ is adjacent to  $v_1$. With the same argument, we have
	\[v_{\ell-2},\overleftarrow{P}[v_{\ell-2},v_1],v_1,v_{\ell-1},\overrightarrow{P}[v_{\ell-1},v_m],v_m\]
	is a path of length $m-1$ with endpoints $v_{\ell-2}$ and $v_m$. Then $d_G(v_{\ell-2})<\widetilde\alpha(G)$. Since $\max\{d_{G}(x),$ $d_G(y)\}\geq \widetilde{\alpha}(G)$ for any $x,y\in V(G)$ with $d_G(x,y)=2$, $v_{\ell-2}$ is adjacent to $v_1$.
	
	Repeating the analysis process, for each $i$ with $1\leq i\leq \ell-1$, we have $d_G(v_i)<\widetilde\alpha(G)$, and $v_i$ is not adjacent to $v_m$ because $G$ has no cycle of length $m$. Moreover, since $\max\{d_G(x),d_G(y)\}\ge \widetilde\alpha(G)$ for any nonadjacent vertices $x$ and $y$ with $d_G(x,y)=2$ and $d_G(v_i,v_j)\le 2$ for any $2\le i< j\le \ell - 1$, we have
	\begin{equation*}
		G[\{v_1,v_2,...,v_{\ell-1}\}]
	\end{equation*}
	is a clique. Recall that $G$ is 2-connected. We have
	$$[\{v_1,v_2,...,v_{\ell-1}\}, \{v_{\ell +1},v_{\ell+2},..., v_m\}]\ne \emptyset.$$
	Hence there exist two integers $j$ and $j^{\prime}$ such that
	$v_j$ is adjacent to $v_{j'}$, where $2\leq j\leq \ell-1$ and $\ell+1\leq j'\leq m-1$. Note that
	\[v_{j'-1},\overleftarrow{P}[v_{j'-1},v_{j+1}],v_{j+1},v_1,\overrightarrow{P}[v_1,v_j],v_j,v_{j'},\overrightarrow{P}[v_{j'},v_m],v_m\]
	is a path of length $m-1$ with endpoints $v_{j'-1}$ and $v_m$.  According to the choice of $P$, $d_G(v_{j'-1})\le d_G(v_1)<\widetilde\alpha(G)$. Recall that $d_G(v_i)<\widetilde\alpha(G)$ for each $1\leq i\leq \ell-1$. Then  $v_{j'-1}$  is adjacent to $v_j$ as $d_G(v_{j'-1},v_j)\le 2$. It follows that $v_{j'-1}$ is adjacent to $v_1$ as $d_G(v_{j'-1},v_1)\le 2$. Then $d_G(v_{j'-1},v_2)\le 2$; hence $v_{j'-1}$ is adjacent to $v_2$. With the same argument, we have
	$$G[\{v_1,v_2,...,v_{\ell-1}\}\cup \{v_{j'-1}\}]$$
	is a clique. Combining with the choice of $\ell$, $j'=\ell+1$. Now
	\[v_\ell,\overleftarrow{P}[v_\ell,v_{j+1}],v_{j+1},v_1,\overrightarrow{P}[v_1,v_j],v_j,v_{\ell+1},\overrightarrow{P}[v_{\ell+1},v_m],v_m\]
	is a path of length $m-1$ with endpoints $v_\ell$ and $v_m$. By the choice of $P$, we have
	$$d_G(v_\ell)\le d_G(v_1)\leq \ell-1<\ell\le d_G(v_\ell),$$
	a contradiction. This proves Claim~\ref{vd}.\end{proof}

Since $1\le s<\widetilde\alpha(G)$, by Claim~\ref{vd}, there exists an integer $k$ where $2\leq k\leq m-1$, such that $|N_G(v_1)\cap \{v_i: 2\leq i\leq k\}|=s$. Let $S_1=N_G(v_1)\cap \{v_i: 2\leq i\leq k\}$, $S_2= N_G(v_1)\cap \{v_i:k+1\leq i\leq m-1\}$, $T_1=N_G(v_m)\cap \{v_j: k\leq j\leq m-1\}$ and $T_2=N_G(v_m)\cap \{v_j: 2\leq j\leq k-1\}$. Hence we have
$$N_G(v_1)=S_1\cup S_2~\text{and}~N_G(v_m)=T_1\cup T_2.$$
Since $G$ has no cycle of length $m$, it follows that
\begin{equation}\label{ST}
	[S_1^-, T_1^+]=\emptyset.
\end{equation}
Recall that $|S_1|=s$, then $|S_1|=|S_1^-|=s$. According to  (\ref{ST}), we have that $|T_1^+|=|T_1|\le t-1$. Since $d_G(v_m)\ge d_G(v_1)\ge \widetilde\alpha(G)$, we have
\begin{equation}\label{T2}
	|T_2| = d_G(v_m)-|T_1|\ge \widetilde\alpha(G)-(t-1)=s.
\end{equation}

Moreover, we have
\begin{equation}\label{SvT}
	[S_2^+\cup \{v_1\}, T_2^+]=\emptyset.
\end{equation}
In fact, if $[\{v_1\}, T_2^+]\neq\emptyset$, then there exists $v_j\in T_2$ such that $v_j^+$ is adjacent to $v_1$. Note that
\[v_1,\overrightarrow{P}[v_1,v_j],v_j,v_m,\overleftarrow{P}[v_m,v_j^+],v_j^+,v_1\]
is a cycle of length $m$, a contradiction.

If $[S_2^+, T_2^+]\neq\emptyset$, there exist $v_{j'}\in S_2$ and $v_{j''}\in T_2$ such that $v_{j'}^+$ is adjacent to
$v_{j''}^+$. Then
\[v_1,\overrightarrow{P}[v_1,v_{j''}],v_{j''},v_m,\overleftarrow{P}[v_m,v_{j'}^+],v_{j'}^+,v_{j''}^+,\overrightarrow{P}[v_{j''}^+,v_{j'}],v_{j'},v_1\]
is a cycle of length $m$, a contradiction.

By (\ref{T2}) and (\ref{SvT}), we have
$$|S_2^+|\le (t-1)-1=t-2.$$
Then $d_G(v_1)=|S_1|+|S_2|\le s+t-2=\widetilde\alpha(G)-1$, which contradicts $d_G(v_1)\ge \widetilde\alpha(G)$. This completes Theorem~\ref{hamiltonian theorem}. \end{proof}

\section{Proof of Theorem \ref{hamiltonian-connected theorem}}\label{sec3}

\begin{proof}[\bf Proof of Theorem \ref{hamiltonian-connected theorem}]  
We say a graph $G$ is {\em admissible} if $\max \{d_G(x), d_G(y)\}\ge \widetilde{\alpha}(G)+1$ for any  nonadjacent vertices $x$ and $y$ with $d_G(x,y)=2$.  Let $G$ be a counterexample to Theorem~\ref{hamiltonian-connected theorem} of order $n$. Subject to this condition, choose $G$ has the maximum size.  That is, for any edge $e$, either $G+e$ is hamiltonian-connected, or $G+e$ is not admissible.  Moreover, $G$ is not complete and $\widetilde{\alpha}(G)\geq 2$. Denote $V^*=\{v\in V(G):d_G(v)\geq \widetilde{\alpha}(G)+1\}$.
 \begin{claim}\label{not-clique}
	$G[V^{*}]$ is not a clique.
\end{claim}
\begin{proof}[Proof of Claim~\ref{not-clique}]
	To the contrary, suppose $G[V^{*}]$ is a clique. We assert every component of $G-V^{*}$ is a clique. Let $D_1, D_2, \ldots,  D_{m}$ be the connected components of $G-V^{*}$. In fact, if some component $D_i$ of $G-V^{*}$ is not a clique, then there exists two nonadjacent vertices $x,y$ at distance two in $D_i$, contradicting the condition $\max \{d_G(x),d_G(y)\}\ge \widetilde{\alpha}(G)+1$.

	First, consider $m=1$. If $|V(D_1)|\ge 3$ and $|V^*|\ge 3$, since $G$ is 3-connected, there exists a matching of cardinality three between $D_1$ and $V^{*}$. If $|V(D_1)|\le 2$ or $|V^*|\le 2$, then $|N_{G[V^{*}]}(V(D_1))|\ge 3$ or  $|N_{D_1}(V^*)|\ge 3$.   In any case, $G$ is hamiltonian-connected.
	
	Next, consider $m\ge 2$. For any $x\in V(D_i)$ and $y\in V(D_j)$, where $i\neq j$, we have
	\begin{equation}\label{neighbor-emptyset}
	N_{G[V^{*}]}(x)\cap N_{G[V^{*}]}(y)=\emptyset.
	\end{equation}

In fact, assume $N_{G[V^{*}]}(x)\cap N_{G[V^{*}]}(y)\neq \emptyset$. Since  $x\in V(D_i)$ and $y\in V(D_j)$ where $i\neq j$, we have $x$ is nonadjacent to $y$, $d_{G}(x)\le \widetilde{\alpha}(G)$ and $d_{G}(y)\le \widetilde{\alpha}(G)$. Then $d_G(x,y)=2$, contradicting $\max \{d_G(x),d_G(y)\}\ge \widetilde{\alpha}(G)+1$. Moreover, $|N_{G[V^*]}(V(D_i))| \allowbreak \ge 3$ for $1\le i\le m$ as $G$ is 3-connected. By (\ref{neighbor-emptyset}), we have $|V^*|\ge 3m$.

	For all $1\le i\le m$, if $|V(D_i)|\ge 3$, then there exists a matching of cardinality three between $V(D_i)$ and $V^{*}$. In this case, we distinguish the following three cases: 1) $x, y \in V^*$, 2) $x \in D_i$ and $y \in D_j$ with $i \ne j$, and 3) $x \in V^*$ and $y \in D_i$ (or vice versa). Since $G[V^*]$ is a clique and there exists a matching of cardinality three between $V(D_i)$ and $V^{*}$, by (\ref{neighbor-emptyset}), it is easy to see that $G$ is hamiltonian-connected.

With the same grgument, if there exists a $D_i$ such that $|V(D_i)|\leq 2$, then $|N_{G[V^{*}]}(V(D_i))|\ge 3$. By (\ref{neighbor-emptyset}) and $|V^{*}|\geq 3m$, we can easily see that $G$ is hamiltonian-connected. This proves Claim~\ref{not-clique}.
\end{proof}

By Claim~~\ref{not-clique}, $G$ contains two nonadjacent vertices with degree at least $\widetilde{\alpha}(G)+1$. We choose $e=uv$ in $E(\overline{G})$ such that $\min \{d_G(u), d_G(v)\}$ is maximized. This means $u,v\in V^*$ and $\min \{d_G(u), d_G(v)\}\geq \widetilde{\alpha}(G)+1$. Noting that adding edges does not increase the bipartite independence number, we have the following claim.

\begin{claim}\label{e}
	$G+e$ is hamiltonian-connected.
\end{claim}

\begin{proof}[Proof of Claim~\ref{e}]
	In $G+e$, every path of length two that contains the edge $e$ has one endpoint in $\{u, v\}$. Since $\min \{d_G(u), d_G(v)\}\geq \widetilde{\alpha}(G)+1$, we have $\max\{d_{G+e}(x_1),d_{G+e}(y_1)\}\geq \widetilde{\alpha}(G)+1\ge \widetilde{\alpha}(G+e)$+1 for any nonadjacent vertices $x_1$ and $y_1$ with $d_{G+e}(x_1,y_1)=2$. This means $G+e$ is admissible, which implies that $G+e$ must be hamiltonian-connected.
\end{proof}

 Since $G$ is a counterexample, there exist two distinct vertices $x$ and $y$ such that $G$ has no Hamilton $(x,y)$-path. By Claim~\ref{e}, let $P$ be a Hamilton $(x,y)$-path in $G+e$. Assume $P=v_1v_2,...,v_n$, where $v_1=x$ and $v_n=y$, and the edge $e=uv=v_kv_{k+1}$ with $d_G(v_{k+1})\ge d_G(v_k)\ge \widetilde{\alpha}(G)+1$.

Let $s$ be an integer such that $1\leq s\leq t$ and $\widetilde\alpha(G)+1=s+t$. Since $\widetilde\alpha(G)\ge 2$, $1\le s\le \frac{\widetilde\alpha(G)+1}{2}<\widetilde\alpha(G)$. Since $s< \widetilde\alpha(G)$ and $d_G(v_k) \geq \widetilde{\alpha}(G) + 1$, there exists an integer $r$ where $1\leq r\leq m$, such that $|N_G(v_k)\cap \{v_i: 1\leq i\leq r\}|=s$. We choose $r$ to be as small as possible, which implies that $v_r\in N_G(v_k)$. Let $S_1=N_G(v_k)\cap \{v_i: 1\leq i\leq r\}$. Then $|S_1|=s$. Moreover, since $v_k$ is not adjacent to $v_{k+1}$, it implies that either $1\leq r\leq k-1$ or $k+2\leq r\leq n$. Hence we consider two cases in the following.

\begin{case}\label{subcase1.1}
	$1\leq r\leq k-1$.
\end{case}

Let $T_1=N_G(v_{k+1})\cap \{v_j: r+1\leq j\leq k-1\}$ and $R_1=N_G(v_{k+1})\cap \{v_j: k+2\leq j\leq n\}$.

\begin{claim}\label{T1R1}
	$|T_1\cup R_1|\le t-1$.
\end{claim}

\begin{proof}[Proof of Claim~\ref{T1R1}] Suppose, to the contrary, that $|T_1\cup R_1|\ge t$. Then $|T_1^+\cup R_1^-|\ge t$, and by $\widetilde{\alpha}(G)=s+t-1$, we have
	\begin{equation*}
		[S_1^+,T_1^+\cup R_1^-]\neq \emptyset.
	\end{equation*}
	If $[S_1^+,T_1^+]\neq \emptyset$, then there exist $v_j\in S_1$ and $v_{j'}\in T_1$ such that $v_j^+$ is adjacent to $v_{j'}^+$. Then
	\[v_1,\overrightarrow{P}[v_1,v_j],v_j,v_k,\overleftarrow{P}[v_k,v_{j'}^+],v_{j'}^+,v_j^+,\overrightarrow{P}[v_j^+,v_{j'}],v_{j'},v_{k+1},\overrightarrow{P}[v_{k+1},v_n],v_n\]
	is a Hamilton $(x,y)$-path in $G$, a contradiction (see Fig. \ref{fig1}(a)).
	
	\medskip
	
	If $[S_1^+,R_1^-]\neq \emptyset$, then there exist $v_j\in S_1$ and $v_{j'}\in R_1$ such that $v_j^+$ is adjacent to $v_{j'}^-$. Hence
	\[v_1,\overrightarrow{P}[v_1,v_j],v_j,v_k,\overleftarrow{P}[v_k,v_j^+],v_j^+,v_{j'}^-,\overleftarrow{P}[v_{j'}^-,v_{k+1}],v_{k+1},v_{j'},\overrightarrow{P}[v_{j'},v_n],v_n\]
	is a Hamilton $(x,y)$-path in $G$, a contradiction (see Fig. \ref{fig1}(b)). This proves Claim~\ref{T1R1}.
	
	\end{proof}

Let $S_2=N_G(v_k)\cap \{v_i: r+1\leq i\leq k-1\}$, $U_2=N_G(v_k)\cap \{v_i: k+2\leq i\leq n\}$ and $T_2=N_G(v_{k+1})\cap \{v_j: 2\leq j\leq r\}$. Since $d_G(v_{k+1})\ge d_G(v_k)\ge \widetilde\alpha(G)+1$, by Claim~\ref{T1R1}, we have
\begin{equation*}
	|T_2|=d_G(v_{k+1})-|T_1\cup R_1|\ge \widetilde\alpha(G)+1-(t-1)=s+1>s.
\end{equation*}

\begin{claim}\label{S2Z2}
	$|S_2\cup U_2|\le t-1$.
\end{claim}

\begin{proof}[Proof of Claim~\ref{S2Z2}] Assume $|S_2\cup U_2|\ge t$. That is, $|S_2^+\cup U_2^-|\ge t$; and by $\widetilde{\alpha}(G)=s+t-1$, we have
	\begin{equation*}
		[S_2^+\cup U_2^-,T_2^-]\neq \emptyset.
	\end{equation*}
	By applying the same argument presented in Claim~\ref{T1R1}. If $[S_2^+, T_2^-]\neq \emptyset$, then there exist $v_j\in S_2$ and $v_{j'}\in T_2$ such that $v_j^+$ is adjacent to $v_{j'}^-$. But now
	\[v_1,\overrightarrow{P}[v_1,v_{j'}^-],v_{j'}^-,v_{j}^+,\overrightarrow{P}[v_j^+,v_k],v_k,v_j,\overleftarrow{P}[v_{j},v_{j'}],v_{j'},v_{k+1},\overrightarrow{P}[v_{k+1},v_n],v_n\]
	is a Hamilton $(x,y)$-path in $G$, a contradiction (see Fig. \ref{fig1}(c)).
	
	\medskip
	If $[U_2^-, T_2^-]\neq \emptyset$, then there exist $v_j\in U_2$ and $v_{j'}\in T_2$ such that $v_j^-$ is adjacent to $v_{j'}^-$. It follows that
	\[v_1,\overrightarrow{P}[v_1,v_{j'}^-],v_{j'}^-,v_{j}^-,\overleftarrow{P}[v_j^-,v_{k+1}],v_{k+1},v_{j'},\overrightarrow{P}[v_{j'},v_k],v_k,v_j,\overrightarrow{P}[v_j,v_n],v_n\]
	is a Hamilton $(x,y)$-path in $G$, a contradiction (see Fig. \ref{fig1}(d)). Hence $|S_2\cup U_2|\le t-1.$ This proves Claim~\ref{S2Z2}.\end{proof}
\medskip

By Claim~\ref{S2Z2}, we have
\begin{equation*}\label{eq3}
	d_G(v_k)=|S_1|+|S_2\cup U_2|\le t-1+s=\widetilde\alpha(G),
\end{equation*}
which is a contradiction.
\begin{figure}[H]
	\begin{center}
		\includegraphics[width=0.9\linewidth]{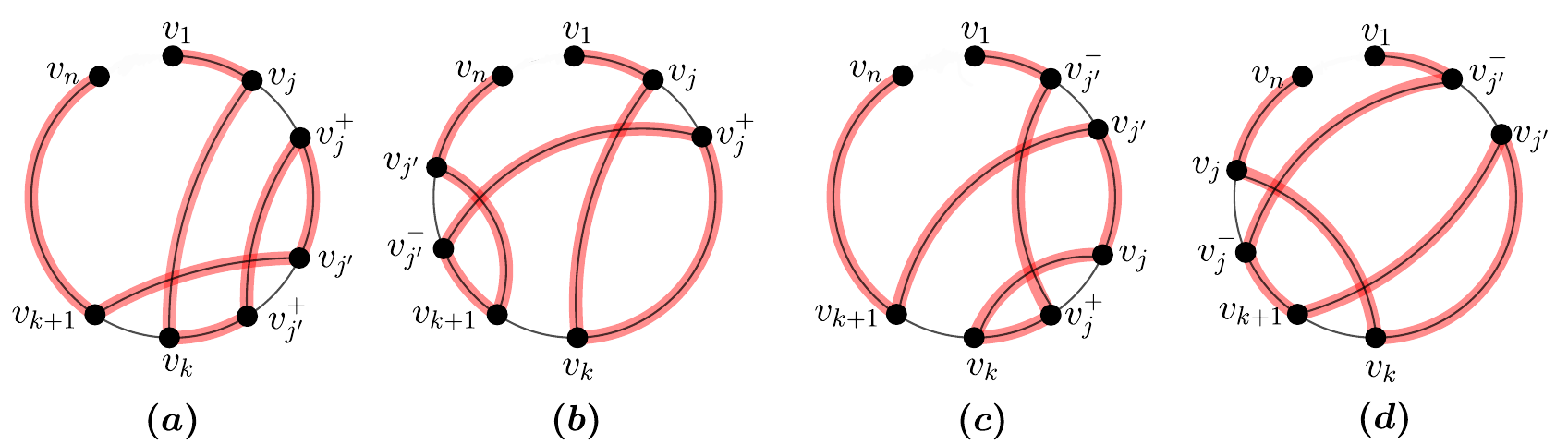}
			\end{center}
		\caption{Illustration of the configurations in Case \ref{subcase1.1}.}
		\label{fig1}
	\end{figure}
	\begin{figure}[H]
	\begin{center}
		\includegraphics[width=0.9\linewidth]{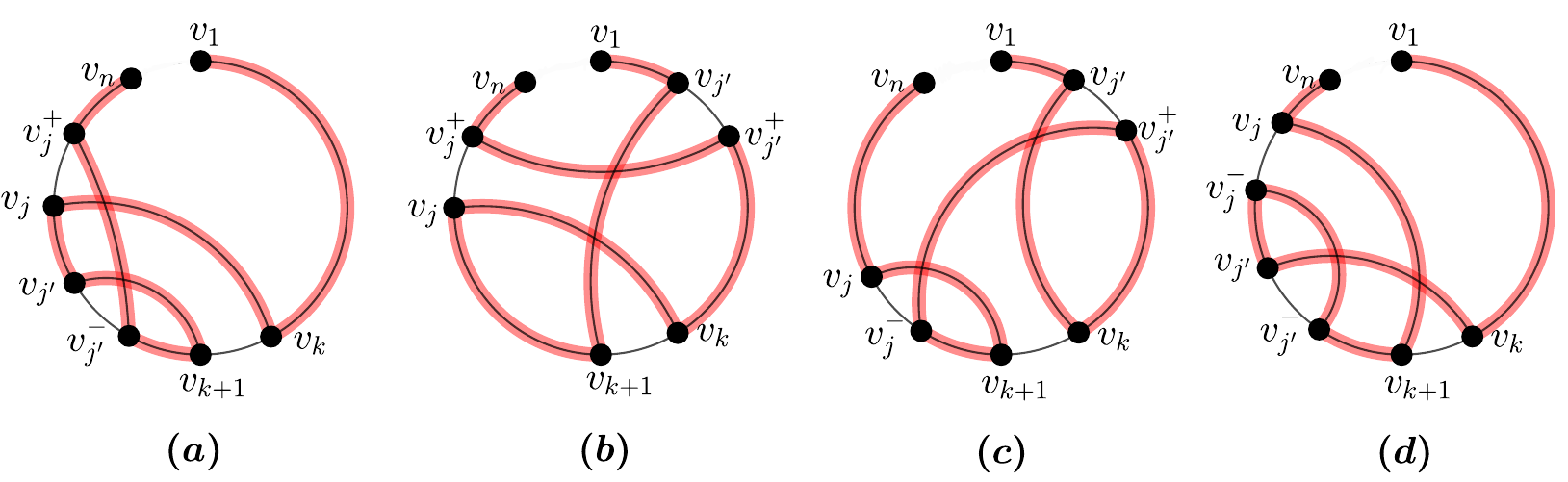}
	\end{center}
		\caption{Illustration of the configurations in Case \ref{subcase1.2}.}
		\label{fig2}
	\end{figure}
Now we consider the other case.

\begin{case}\label{subcase1.2}
	$k+2\leq r\leq n$.
\end{case}

Denote $S_3=N_G(v_k)\cap \{v_i: r\leq i\leq n\}$. Since $|S_1|=s$ and $v_k$ is adjacent to $v_r$, we have
\begin{equation}\label{eq4}
	|S_3|=d_G(v_k)-|S_1|+1\ge \widetilde\alpha(G)+1-s+1=t+1.
\end{equation}

Then there exists an integer $r \leq r'\leq n$ such that $|N_G(v_k)\cap \{v_i: r'\leq i\leq n\}|=s+1$, and we choose $r'$ to be the maximum possible value. It implies that $v_k$ is adjacent to $v_{r'}$. Let $U_3=N_G(v_k)\cap \{v_i: r'\leq i\leq n-1\}$, $T_3=N_G(v_{k+1})\cap \{v_j: k+2\leq j\leq r'-1\}$ and $R_3=N_G(v_{k+1})\cap \{v_j: 1\leq j\leq k-1\}$.

\begin{claim}\label{T3R3}
	$|T_3\cup R_3|\le t-1$.
\end{claim}

\begin{proof} [Proof of Claim~\ref{T3R3}] Suppose, to the contrary, that $|T_3\cup R_3|\ge t$. Then $|T_3^-\cup R_3^+|\ge t$. Recall that $|U_3|=|U_3^+|\ge s$, by $\widetilde{\alpha}(G)=s+t-1$, we have
	\begin{equation}\label{eq5}
		[U_3^+, T_3^-\cup R_3^+]\neq \emptyset.
	\end{equation}
	If $[U_3^+,T_3^-]\neq\emptyset$, then there exist $v_j\in U_3$ and $v_{j'}\in T_3$ such that $v_j^+$ is adjacent to $v_{j'}^-$. But now
	\[v_1,\overrightarrow{P}[v_1,v_k],v_k,v_j,\overleftarrow{P}[v_j,v_{j'}],v_{j'},v_{k+1},\overrightarrow{P}[v_{k+1},v_{j'}^-],v_{j'}^-,v_j^+,\overrightarrow{P}[v_j^+,v_n],v_n\]
	is a Hamilton $(x,y)$-path in $G$, a contradiction (see Fig. \ref{fig2}(a)).
	\medskip
	
	If $[U_3^+,R_3^+]\neq\emptyset$, then there exist $v_j\in U_3$ and $v_{j'}\in R_3$ such that $v_j^+$ is adjacent to $v_{j'}^+$. Then
	\[v_1,\overrightarrow{P}[v_1,v_{j'}],v_{j'},v_{k+1},\overrightarrow{P}[v_{k+1},v_j],v_j,v_k,\overleftarrow{P}[v_k,v_{j'}^+],v_{j'}^+,v_j^+,\overrightarrow{P}[v_j^+,v_n],v_n\]
	is a Hamilton $(x,y)$-path in $G$, a contradiction (see Fig. \ref{fig2}(b)). Hence we have $|T_3\cup R_3|\le t-1$. This proves Claim~\ref{T3R3}.\end{proof}

Now, we denote $T_4=N_G(v_{k+1})\cap \{v_j: r'\leq j\leq n\}$. Recall that $d_G(v_{k+1})\ge d_G(v_k)\ge \widetilde\alpha(G)+1$, by Claim~\ref{T3R3} and $\widetilde{\alpha}(G)=s+t-1$, we have
\begin{equation*}\label{eq6}
	|T_4|=d_G(v_{k+1})-|T_3\cup R_3|\ge \widetilde\alpha(G)+1-(t-1)=s+1.
\end{equation*}
It follows that there exists an integer $r'\leq r''\leq n$ such that $|N_G(v_{k+1})\cap \{v_j: r''\leq j\leq n\}|=s$, and we choose $r''$ to be maximum. Then $v_{k+1}$ is adjacent to $v_{r''}$.

Let $R_4=N_G(v_{k+1})\cap \{v_j: r''\leq j\leq n\}$, $S_4=N_G(v_k)\cap \{v_i: 1\leq i\leq k-1]\}$ and $U_4=N_G(v_k)\cap \{v_i:k+2\leq i\leq r'\}$.

\begin{claim}\label{S4U4}
	$|S_4\cup U_4|\le t-1$.
\end{claim}

\begin{proof}[Proof of Claim~\ref{S4U4}] Suppose that $|S_4\cup U_4|\ge t$. That is, $|S_4^+\cup U_4^-|\ge t$. Recall that $|R_4^-|=|R_4|=s$. By $\widetilde{\alpha}(G)=s+t-1$, we have
	\begin{equation*}\label{eq7}
		[R_4^-,S_4^+\cup U_4^-]\neq \emptyset.
	\end{equation*}
	
	If $[R_4^-, S_4^+]\neq \emptyset$, then there exist $v_j\in R_4$ and $v_{j'}\in S_4$ such that $v_j^-$ is adjacent to $v_{j'}^+$. It follows that
	\[v_1,\overrightarrow{P}[v_1,v_{j'}],v_{j'},v_{k},\overleftarrow{P}[v_{k},v_{j'}^+],v_{j'}^+,v_j^-,\overleftarrow{P}[v_j^-,v_{k+1}],v_{k+1},v_j,\overrightarrow{P}[v_j,v_n],v_n\]
	is a Hamilton $(x,y)$-path in $G$, a contradiction (see Fig. \ref{fig2}(c)).
	
	\medskip
	
	If $[R_4^-,U_4^-]\neq\emptyset$, then there exist $v_j\in R_4$ and $v_{j'}\in U_4$ such that $v_j^-$ is adjacent to $v_{j'}^-$. Hence
	\[v_1,\overrightarrow{P}[v_1,v_k],v_k,v_{j'},\overrightarrow{P}[v_{j'},v_j^-],v_j^-,v_{j'}^-,\overleftarrow{P}[v_{j'}^-,v_{k+1}],v_{k+1},v_j,\overrightarrow{P}[v_j,v_n],v_n\]
	is a Hamilton $(x,y)$-path in $G$, a contradiction (see Fig. \ref{fig2}(d)). Therefore, $|S_4\cup U_4|\le t-1$. This proves Claim~\ref{S4U4}.\end{proof}

According to Claim~\ref{S4U4}, we have
\begin{equation*}
	d_G(v_k)=|S_4\cup U_4\cup (N_G(v_k)\cap \{v_i:  r'\leq i\leq n\})|-1\le t-1+s+1-1=\widetilde\alpha(G).
\end{equation*}
This contradicts the fact that $d_G(v_k)\ge \widetilde\alpha(G)+1$.  Thus, we have completed the proof of Theorem~\ref{hamiltonian-connected theorem}.
\end{proof}

\section*{Acknowledgment}
We would like to thank  Kun Cheng for drawing our attention to an error in the preliminary version. We also thankful to Dr.  Feng Liu for the inspiring discussions on this subject.

\section*{Declarations}
\noindent$\textbf{Conflict~of~interest}$
The authors declare that they have no known competing financial interests or personal relationships that could have appeared to influence the work reported in this paper.

\noindent$\textbf{Data~availability}$
Data sharing not applicable to this paper as no datasets were generated or analysed during the current study.


\begin{thebibliography}{9999}
	
	\bibitem{Bondy2008}
	J.A. Bondy and U.S.R. Murty,
	\newblock Graph theory. Graduate texts in mathematics, vol. 244, Springer, 2008.
	
	\bibitem{Chen2022}
	M. Chen,
	\newblock Hamilton-connected, vertex-pancyclic and bipartite holes,
	\newblock \emph{Discrete Math.}, \textbf{345} (2022) 113158.
	
	\bibitem{Chvatal1972}
	V. Chv\'{a}tal and P. Erd\H{o}s,
	\newblock A note on hamiltonian circuits,
	\newblock \emph{Discrete Math.}, \textbf{2} (1972) 111--113.
	
	
	
	\bibitem{Dirac1952}
	G.A. Dirac,
	\newblock Some theorems on abstract graphs,
	\newblock \emph{Proc. Lond. Math. Soc.}, \textbf{3} (1952) 69--81.
	
	
	\bibitem{Draganic2024}
	N. Dragani\'{c}, D.M. Correia and B. Sudakov,
	\newblock A generalization of Bondy's pancyclicity theorem,
	\newblock \emph{Comb. Prob. Comput.}, \textbf{33} (2024) 554--563.
	
	\bibitem{Fan1984}
	G. Fan,
	\newblock New sufficient conditions for cycles in graphs,
	\newblock \emph{J. Combin. Theory Ser. B}, \textbf{37} (1984) 221--227.
	
	
	\bibitem{Faudree1989}
	R.J. Faudree, R.J. Gould, M.S. Jacobson and R.H. Schelp,
	\newblock Neighborhood unions and hamiltonian properties in graphs,
	\newblock \emph{J. Combin. Theory Ser. B}, \textbf{47} (1989) 1--9.
	
	
	\bibitem{Gould2014}
	R.J. Gould,
	\newblock Recent advances on the hamiltonian problem: Survey III,
	\newblock \emph{Graphs Combin.}, \textbf{30} (2014) 1--46.
	
	
	\bibitem{Han2024}
	J. Han, J. Hu, L. Ping, G. Wang, Y. Wang and D. Yang,
	\newblock Spanning trees in graphs without large bipartite holes,
	\newblock \emph{Graphs Combin.}, \textbf{33} (2024) 270--285.
	
	
	
	
	
	\bibitem{Li2013}
	H. Li,
	\newblock Generalizations of Dirac's theorem in hamiltonian problem-A survey,
	\newblock \emph{Discrete Math.}, \textbf{313} (2013) 2034--2053.
	
	
	\bibitem{Li2025} C. Li and F. Liu, \newblock An Ore-type condition for hamiltonicity in graphs, \newblock (2025) arXiv: 2504.04493.
	
	\bibitem{Mcdiarmid2017}
	C. McDiarmid and N. Yolov,
	\newblock Hamilton cycles, minimum degree, and bipartite holes,
	\newblock \emph{J. Graph Theory}, \textbf{86} (2017) 277--285.
	
	
	\bibitem{Ore1960}
	O. Ore,
	\newblock Note on Hamilton circuits,
	\newblock \emph{Amer. Math. Monthly}, \textbf{67} (1960) 55.
	
	\bibitem{Ore1963}
	O. Ore,
	\newblock Hamilton connected graphs,
	\newblock \emph{J. Math. Pures Appl.}, \textbf{42} (1963) 21--27.
	
	
	
	\bibitem{West1996}
	D.B. West,
	Introduction to Graph Theory, Prentice Hall, Inc., 1996.
	
	\bibitem{Zhou2024}
	Q. Zhou, H. Broersma, L. Wang and Y. Lu,
	\newblock A note on minimum degree, bipartite holes, and hamiltonian properties,
	\newblock \emph{Discuss. Math. Graph Theory}, \textbf{44} (2024) 717--726.
	
	
	
\end{thebibliography}
\end{document}